\documentclass[12pt]{article}
\usepackage[T2A]{fontenc}
\usepackage[utf8]{inputenc}
\usepackage{amssymb,amsmath,amsfonts,amsthm,amscd,latexsym,verbatim,graphics,epsfig,indentfirst,xcolor,array,ulem,textcomp}
\usepackage{geometry}
\def\id{\mathrm{id}}

\def\conj{\mathrm{conj}}

\def\Aut{\mathrm{Aut}}

\def\Inn{\mathrm{Inn}}
\def\Imm{\mathrm{Im}}

\def\A{\mathrm{A}}

\def\Ss{\mathrm{S}}
\def\PSL{\mathrm{PSL}}
\def\PGL{\mathrm{PGL}}
\def\SL{\mathrm{SL}}

\newtheorem{lemma}{Lemma}
\newtheorem{proposition}{Proposition}
\newtheorem{theorem}{Theorem}
\newtheorem{corollary}{Corollary}
\newtheorem{remark}{Remark}
\newtheorem{example}{Example}
\newtheorem{definition}{Definition}

\begin{document}

\sloppy

\begin{center}
{\Large
Rota---Baxter operators on dihedral and alternating groups}

\smallskip

Alexey Galt, Vsevolod Gubarev
\end{center}

\begin{abstract}
Rota---Baxter operators on algebras, which appeared in 1960, have connections with different versions of the Yang---Baxter equation, pre- and postalgebras, double Poisson algebras, etc. 
In 2020, the notion of Rota---Baxter operator on a~group was defined by L. Guo, H. Lang, Yu. Sheng.

In 2023, V. Bardakov and the second author showed that all Rota---Baxter operators on simple sporadic groups are splitting, i.\,e. they are defined via exact factorizations.
In the current work, we clarify for which $n$, there exist non-splitting Rota---Baxter operators on the alternating group~$\A_n$. 
For the corresponding~$n$, we describe all non-splitting Rota---Baxter operators on~$\A_n$.
Moreover, we describe Rota---Baxter operators on dihedral groups~$D_{2n}$ providing the general construction which lies behind all non-splitting Rota---Baxter operators on $\A_n$ and $D_{2n}$.

\end{abstract}

\section{Introduction}

In 1960~\cite{Baxter}, G. Baxter introduced the notion of Rota---Baxter operator on an algebra~$A$ as a linear operator~$R$ satisfying the relation
\begin{equation} \label{RBOnAlgebra}
R(x)R(y) = R( R(x)y + xR(y) + \lambda xy ),
\end{equation}
where $x,y\in A$ and the scalar $\lambda$ from the ground field~$F$ is fixed, and it is called a weight of~$R$.

The identity~\eqref{RBOnAlgebra} for $\lambda = 0$ may be considered a generalization of integration by parts formula.
When $\lambda\neq0$, the canonical example is a projection of the whole algebra~$A$ on its subalgebra~$A_1$ provided a decomposition $A = A_1\dotplus A_2$ of~$A$ into a direct vector-space sum of its subalgebras~$A_1,A_2$.

After more than 60 years of the study, a lot of connections and applications of Rota---Baxter operators were found: with different versions of the Yang---Baxter equation~\cite{Aguiar00,BelaDrin82,Semenov83}, pre- and postalgebras~\cite{Aguiar00,BBGN2011,GuoMonograph}, double Poisson algebras~\cite{DoubleLie2,DoubleLie,DoublePoissonFree}, etc.

Recently, in 2020 (published in 2021), the analogous notion of Rota---Baxter operator (of weight~$\pm1$) on groups was suggested by L. Guo, H. Lang, and Yu. Sheng~\cite{Guo2020}.
Since this pioneer article appeared, a dozen papers have been published, and the direction of Rota---Baxter groups is rapidly developing.
On the one hand, the correspondence between Rota---Baxter operators on Lie groups and 
Rota---Baxter operators on Lie algebras is studied~\cite{Guo2020,Jiang}.
On the other hand, the definition of Rota---Baxter operator on groups was extended on Clifford semigroups~\cite{Catino2023}, Hopf algebras~\cite{RBHopf}, and others.
There is a~deep connection~\cite{BG2} between Rota---Baxter operators on groups and so called skew left braces~\cite{GV2017,Problems}, which, in turn, serve as a source of set-theoretical solutions to the quantum Yang---Baxter equation.
The cohomology theory for Rota---Baxter (Lie) groups was presented in~\cite{Das,Jiang}.

Given a group~$G$, one has always trivial RB-operators $B_e\colon g \to e$ and $B_{-1}\colon g\to g^{-1}$. The analog of the projection operator mentioned above for algebras is the following:
given an exact factorization $G = HL$, the map
$$
B\colon hl \to l^{-1}
$$
is a~Rota---Baxter operator on~$G$. We will call it splitting RB-operator.
The criterion, when a given RB-operator is a~splitting one, was stated in~\cite{BG}.
It actually says that $R$ is splitting RB-operator on a group~$G$ if and only if $\Imm(B\widetilde{B}) = 1$, where 
$$
\widetilde{B} \colon G\to G, \quad \widetilde{B}(g) = g^{-1}B(g^{-1}) 
$$
is another RB-operator on~$G$ defined with the help of~$B$.
Different constructions of non-splitting Rota---Baxter operators on groups were suggested therein.
However, it occurs that there are no non-splitting RB-operators on all simple sporadic groups~\cite{BG}.

In the current work, we continue the mentioned in~\cite{BG} direction of description of RB-operators on all finite simple groups; here we deal with alternating groups $\A_n$.
The main result is the complete classification of (non-splitting) Rota---Baxter operators on~$\A_n$.
For any RB-operator~$B$ on a~group~$G$, one has a~factorization 
$$
G = \Imm(B)\cdot \Imm(\widetilde{B}).
$$
Thus, we may apply the known description of factorizations of finite simple groups~\cite{LPS_fact}. Further, we arrive at properties of sharply 2-transitive groups and involve the Zassenhaus theorems (Theorems XII.9.7 and XII.9.8 from \cite{HuppertB}).
As a consequence, we show that there are non-splitting RB-operators~$R$ on~$\A_n$ if and only if $n$ satisfies certain conditions, 
and there is an infinite number of such ones.
When $n$ is so, we clarify the general construction, which lies behind~$R$. 

We prove that if $G$ is a finite group, $B$ is an RB-operator on it, $R = \Imm(B\widetilde{B})$ is an abelian subgroup of~$G$, and $G = H\cdot L = \ker(B)\cdot \Imm(B)$, then $B( hl ) = C(l)$ for an RB-operator $C$ on $L$. Moreover, $\Imm(\widetilde{C})$ normalizes~$H$ and $C$ is a~homomorphism from $L$ onto its abelian subgroup isomorphic to~$R$.
We apply this result not only for alternating groups but also for dihedral ones.
We show that for any non-splitting RB-operator~$B$ on $\A_n$, we have $\Imm(B\widetilde{B})\simeq \mathbb{Z}_2$. 
Regarding to dihedral groups~$D_{2n}$, we prove that there are no non-splitting RB-operators on~$D_{2n}$, when $n$ is odd. When $n$ is even, we show that for any non-splitting RB-operator~$B$ on $D_{2n}$, we have 
either $\Imm(B\widetilde{B})\simeq \mathbb{Z}_2$ or $\Imm(B\widetilde{B})\simeq \mathbb{Z}_2\times \mathbb{Z}_2$.

Let us highlight that we suggest a new kind of equivalence on a set of Rota---Baxter operators. In~\cite{BG}, it was noted that given a Rota---Baxter group~$(G,B)$ and $\varphi\in \Aut(G)$, one has again an RB-operator $B^{(\varphi)}$ on $G$, which acts by the formula $B^{(\varphi)}(g) = \varphi^{-1}B\varphi(g)$. Therefore, it is natural to classify Rota---Baxter operators on a group~$G$ up to such action of~$\Aut(G)$. We present an equivalence obtained under the action of a subgroup of~$\Aut(G\times G)$ isomorphic to $(\Aut(G)\times \Inn(G))\rtimes \mathbb{Z}_2$. As far as we know, there is no analogous action for Rota---Baxter operators on algebras. We prove that such equivalence preserves main properties of Rota---Baxter operators, e.\,g. the property of splitting one. Moreover, the mentioned above construction of an RB-operator~$G$ such that $\Imm(B\widetilde{B})$ is abelian, is preserved too.
So, we find all non-splitting Rota---Baxter operators on $\A_n$ up to the equivalence.

The work is organized as follows. 
We provide the required preliminaries about Rota---Baxter groups in~Section 2 and about sharply doubly transitive groups in~Section 3, respectively.
In~Section 4, we define the new equivalence of RB-operators (Proposition~\ref{prop:equivalentRB}).
In~Section 5, we derive the construction of a Rota---Baxter operator~$B$ such that
$R = \Imm(B\widetilde{B})$ is abelian and $G = \ker(B)\cdot \Imm(B)$ (Lemma~\ref{lem:Main}).

In~Section 6, we describe Rota---Baxter operators on dihedral groups (Theorem~\ref{thm:Dihedral}).
In~Section 7, we classify, up to equivalence, Rota---Baxter operators on alternating groups~$\A_n$ (Theorem~\ref{thm:alternating}).
For small~$n$, we use a computer algebra system~GAP to compute RB-operators on~$\A_n$.
We use the reformulation of a Rota---Baxter operator~$B$ in terms of its graph
in $G\times G$~\cite{Jiang} (see Proposition~\ref{prop:RBviaSubgroups} below)
and write a~program with a starting point based on the program from~\cite{Rathee}.

\section{Preliminaries}

\begin{definition}[\cite{Guo2020}]
Let $G$ be a group.
A map $B\colon G\to G$ is called a Rota---Baxter operator of weight~1 if
\begin{equation}\label{RB}
B(g)B(h) = B( g B(g) h B(g)^{-1} )
\end{equation}
holds for all $g,h\in G$.
\end{definition}

Note that a notion of Rota---Baxter operator on an abelian~$G$ coincides with homomorphism.

\begin{proposition}[\cite{Guo2020}]\label{prop:initial}
Let $(G,B)$ be a Rota---Baxter group. Then

(a) $\Imm(B)$ and $\ker(B)$ are subgroups of~$G$,

(b) $B(e) = e$,

(c) $B(g)^{-1} = B(B(g)^{-1}g^{-1}B(g))$ for every $g \in G$,

(d) $B(h) = B(gh)$ for all $h\in G$ and $g\in \ker(B)$.
\end{proposition}

\begin{proposition}[\cite{Guo2020}]\label{prop:Btilde}
Let $(G,B)$ be a Rota---Baxter group. Then

(a) the map $\widetilde{B}$ defined as follows, $\widetilde{B}(g) = g^{-1}B(g^{-1})$ is a Rota---Baxter operator on~$G$,

(b) given $\varphi\in \Aut(G)$, $B^{(\varphi)} = \varphi^{-1}B\varphi$
is a~Rota---Baxter operator on $G$.
\end{proposition}

\begin{proposition}[\cite{Guo2020}]\label{prop:Derived}
Let $(G,\cdot,B)$ be a Rota---Baxter group. Then

(a) The pair $(G, \circ )$ with the product
\begin{equation}\label{R-product}
g\circ h = gB(g)hB(g)^{-1}, \quad g,h\in G,
\end{equation}
is a group.

(b) The operator $B$ is a Rota---Baxter operator on the group $(G,\circ)$.

(c) The map $B\colon (G,\circ) \to (G,\cdot)$
is a homomorphism of Rota---Baxter groups.
\end{proposition}

We will denote the group $(G, \circ )$ as $G^{(\circ)}$.

\begin{proposition}[\cite{Guo2020}]\label{Prop:STSh} 
If $(G, B)$ is an RB-group, then

(a) $\ker(B)$ and $\ker(\widetilde{B})$ are normal in $G^{(\circ)}$,

(b) $\ker(B) \unlhd \Imm(\widetilde{B})$ and $\ker(\widetilde{B}) \unlhd \Imm(B)$ in $G$,

(c) We have the isomorphism of the quotient groups
\begin{equation}\label{FactorIso}
\Imm(\widetilde{B})/\ker(B)\simeq \Imm(B)/\ker(\widetilde{B}),
\end{equation}

(d) We have the factorization
\begin{equation}\label{ImageFactorization}
G = \Imm(\widetilde{B})\Imm(B).
\end{equation}
\end{proposition}

Denote by $B_+$ the map on the group~$G$ defined as follows, $B_+(g) = gB(g)$. 

\begin{proposition}[{\cite[Lemma~5(c)]{BG}}]\label{lem:BBTildeCommute}
Given a Rota---Baxter group $(G,B)$, one has $BB_+ = B_+B$.  
\end{proposition}

\begin{remark}\label{rem:BBTildeCommute}
Because of Proposition~\ref{lem:BBTildeCommute}, one may note
$\Imm(B\widetilde{B}) = \Imm(\widetilde{B}B)$.
Indeed, 
$$
(B\widetilde{B})(g) 
= (BB_+)(g^{-1})
= (B_+B)(g^{-1})
= \widetilde{B}((B(g^{-1})^{-1}).
$$
However, the equality $B\widetilde{B}=\widetilde{B}B$ is not true, see Example~\ref{exm:BBTildeNotCommute} below.   
\end{remark}

Given an integer~$n$, denote by $a^{\circ(n)}$ the $n$th power of $a$ under the product $\circ$. 

\begin{proposition}[{\cite[Proposition~14]{BG}}]\label{prop:CircPower}
Given a Rota---Baxter group $(G,B)$, we have the equality
$a^{\circ(k)} = (B_+(a))^k(B(a))^{-k}$ for all $a\in G$ and $k\in \mathbb{Z}$.   \end{proposition}

Let us deduce an easy but important equality.

\begin{lemma}\label{lem:ImagesManipul}
Let $(G, B)$ be an RB-group. Then
$\Imm(B\widetilde{B}) = \Imm(\widetilde{B}B) = \Imm(B)\cap \Imm(\widetilde{B})$.
\end{lemma}

\begin{proof}
The first equality follows from Remark~\ref{rem:BBTildeCommute} and it implies the inclusion $\Imm(B\widetilde{B})\subseteq \Imm(B)\cap \Imm(\widetilde{B})$.    
Suppose that $x \in \Imm(B)\cap \Imm(\widetilde{B})$, i.\,e. $x = B(g) = hB(h)$ for some $g,h\in G$. By Proposition~\ref{prop:initial}c), we have
$$
h = B(g)B(h)^{-1}
= B(g)B(B(h)^{-1}h^{-1}B(h))
= B(z)
$$
for $z = g\circ B(h)^{-1}h^{-1}B(h)$.
Since $\Imm(B)$ is a subgroup, $h^{-1}\in \Imm(B)$.
Hence, $x = \widetilde{B}(h^{-1}) \in \Imm(\widetilde{B}B)$.
\end{proof}

\begin{example}[{\cite[Lemma~2.6]{Guo2020}}]\label{exm:split}
Let $G$ be a group. Given an exact factorization $G = HL$,
a map $B\colon G\to G$ defined as follows,
$$
B(hl) = l^{-1}
$$
is a Rota---Baxter operator on~$G$.
\end{example}

We call such Rota---Baxter operator on~$G$ a~splitting Rota---Baxter operator.

\begin{proposition}[{\cite[Proposition 16]{BG}}]\label{prop:SplittingCond}
Let $G$ be a group and let $B\colon G\to G$ be an RB-operator on~$G$.
Then $B$ is a splitting Rota---Baxter operator on~$G$
if and only if $\Imm(\widetilde{B}B) {=} 1$.
\end{proposition}

\begin{proposition}[{\cite[Proposition~20]{BG}}] \label{prop:semidirect}
Given a semidirect product $G = H \rtimes L$, let $C$ be a Rota---Baxter
operator on $L$. Then a map $B \colon G \to G$ defined by the formula $B(hl) = C(l)$, where $h \in H$ and $l \in L$, is a Rota---Baxter operator.
\end{proposition}

\begin{corollary}\label{coro:Main}
Given an exact factorization  $G = HL$, let $C$ be a Rota---Baxter
operator on $L$. If $\Imm(\widetilde{C})\leqslant L$ normalizes $H$, then a map $B \colon G \to G$ defined by the formula $B(hl) = C(l)$, where $h \in H$ and $l \in L$, is a Rota---Baxter operator.
\end{corollary}

\begin{proof}
Let $h,h'\in H$ and $l,l'\in L$. Then
\begin{multline*}
B(hlB(hl)h'l'(Bhl)^{-1})
= B(hlC(l)h'l'C(l)^{-1})
= B(hh''lC(l)l'C(l)^{-1}) \\
= C(lC(l)l'C(l)^{-1})
= C(l)C(l')
= B(hl)B(h'l'),
\end{multline*}
where $h'' = h'^{(lC(l))^{-1}} = h^{\widetilde{C}(l^{-1})^{-1}}\in H$.
\end{proof}

\begin{proposition}[{\cite[Proposition~21]{BG}}]\label{prop:RB-Hom}
If $G$ is a group and $H$ is its abelian subgroup,
then any homomorphism $B\colon G \to H$ is a Rota---Baxter operator.
\end{proposition}

\begin{lemma} \label{lem:centerPreimage}
Let $B$ be an RB-operator on a finite group~$G$, $c\in Z(G)$.

(a) If $B$ is invertible and $B(x) = c$,
then $x$ commutes with $\Imm(\widetilde{B})$.

(b) We have that $|B(c)|$ divides $|c|$.
\end{lemma}

\begin{proof}
(a) Since $B$ is an isomorphism from $G^{(\circ)}$ onto $G$, $x\in Z(G^{(\circ)})$, i.\,e. 
$x\circ g = g\circ x$ for every $g\in G$. Hence,
$$
xB(x)gB(x)^{-1}
= xg
= gB(g)xB(g)^{-1}.
$$
Thus, $x = B_+(g)x(B_+(g))^{-1} = x^{(B_+(g))^{-1}}$.
It remains be noticed that $\Imm(\widetilde{B}) = \Imm(B_+)$.

(b) Let $|c| = k$. By~\eqref{RB}, we have 
$$
B(z)(B(c))^k
= B(z)B(c)\cdot (B(c))^{k-1}
= B(zc)(B(c))^{k-1}
= \ldots 
= B(zc^k)
= B(z).
$$ 
Hence, $(B(c))^k = e$.
\end{proof}

\section{Sharply 2-transitive groups}

In this section, we provide necessary information about sharply 2-transitive groups.
The structure of a sharply 2-transitive group is described in the following theorems. 

\begin{theorem}[Zassenhaus, {\cite[Theorem XII.9.7]{HuppertB}}]\label{thm:Zassenhaus-1}
Suppose that $q$ is a prime-power and $m$~is a natural number with the following properties:

(1) Each prime divisor of $m$ divides $q-1$.

(2) If 4 divides $m$, then 4 divides $q-1$.

\noindent Let $GF(q^m)^*=\langle w\rangle$ and suppose that $(m,t)=1$. Let $a,b$ be the following elements of the group of semilinear mappings on $GF(q^m)$:
\begin{equation}\label{eq:Frobenius_group}
xa=w^mx, \quad xb=w^tx^q. 
\end{equation}
Then $N=\langle a,b\rangle$ is a regular permutation group on $GF(q^m)^*$ and $\langle a\rangle$ is an irreducible normal subgroup of $N$. If $F$ is the group of all translations on $GF(q^m)$, $N$ normalizes $F$ and $L(m,q,t) =FN$ is a sharply doubly transitive permutation group of degree $q^m$.
\end{theorem}

\noindent Note that the groups $L(m,q,t)$ are isomorphic for fixed $m,q$ and different admissible $t$.

\begin{theorem}[Zassenhaus, {\cite[Theorem XII.9.8]{HuppertB}}] \label{thm:Zassenhaus-2}
Let $L$ be a sharply doubly transitive permutation group. Then $L$ is a Frobenius group with elementary abelian Frobenius kernel~$F$. To within similarity, $L$ permutes $GF(p^r)$, and $F$ is the group of translations on $GF(p^r)$. If $p^r$ is distinct from $5^2, 7^2, 11^2,23^2, 29^2$, and $59^2$, $L$ is permutation isomorphic to one of the groups $L(m,q,t)$ defined in Theorem~\ref{thm:Zassenhaus-1}.
\end{theorem}

We want to find the conditions when $L$~is contained in the alternating group and $L$ has a subgroup of index~2.

In the remarks after \cite[Theorem~A]{WiegoldW}, the authors say that the sharply 2-transitive group consists of even permutations only if $p=2$, or $r$~is even and $p$~is odd. 
Note that this condition is necessary but not sufficient. 
For example, if $n=13^2$, then the symmetric group $\Ss_{169}$ has two sharply 2-transitive groups $L_1$, $L_2$, where $L_1\simeq(\mathbb{Z}_{13}\times\mathbb{Z}_{13}):\mathbb{Z}_{168}$ and $L_2\simeq(\mathbb{Z}_{13}\times\mathbb{Z}_{13}):(\mathbb{Z}_{21}:\mathbb{Z}_{8})$.
It is clear that the generator of $\mathbb{Z}_{168}$ in $L_1$ and the generator of $\mathbb{Z}_{8}$ in $L_2$ are odd permutations, and so they do not belong to $\A_{169}$.

\begin{corollary}\label{coro:Zassenhaus}
Let $L=FN$ be a sharply doubly transitive permutation group of degree $n=q^m$. If $L\leqslant\A_n$ and $L$ has a subgroup of index~2, then one and only one of the following holds: 
\begin{itemize}
\item[(1)]  $q\equiv3\!\!\pmod4$ and $L=L(m,q,t)$ is defined in Theorem~\ref{thm:Zassenhaus-1}. Moreover, $L$ has two non-isomorphic subgroups of index~2, except $L=L(2,3,1)$ which has only one non-isomorphic subgroup of index~2;

\item[(2)]  $n=7^2$ and $N\simeq 2^-\Ss_4\simeq\SL_2(3).2$. Moreover, $L$ has one non-isomorphic subgroup of index~2;

\item[(3)]  $n=23^2$ and $N\simeq 2^-\Ss_4\times\mathbb{Z}_{11}\simeq\SL_2(3).2\times\mathbb{Z}_{11}$. Moreover, $L$ has one non-isomorphic subgroup of index~2.
\end{itemize}
\end{corollary}

\begin{proof}    
At first, consider the general case $L=L(m,q,t)=FN$. The Frobenius group of even degree does not have subgroups of index~2, and therefore $q$ must be odd.
Since $q$~is odd, we have $F\leqslant \A_n$, where $n=q^m$. It follows from the proof of Theorem~\ref{thm:Zassenhaus-1} that $q^m-1=m(q-1)l$ for some integer~$l$, and the order of~$a$ is equal to $(q-1)l$. Moreover, $b^{-1}ab=a^q$ and $b^m=a^{tl}$. 

Therefore, the element $a$ is the product of $m$ cycles of length $(q-1)l$. 
Thus, $a\in \A_n$ if and only if $m$ is even. 

It follows from the equality~(\ref{eq:Frobenius_group}) that $m\,|\,|b|$. Therefore,
$$
|b| = \frac{m|a|}{(|a|,tl)}
= \frac{m|a|}{l(q-1,t)}
= \frac{q^m-1}{l(q-1,t)}
= \frac{n-1}{l(q-1,t)}.
$$
Put $s=l(q-1,t)$. The element $b$ is the product of $s$ cycles of length $(n-1)/s$. 
Thus, $b\in \A_n$ if and only if $s$~is even. 
Since $(m,t)=1$ and $m$ must be even, we have $t$ is odd. Therefore, $s$ is even if and only if $l$ is even. 

We obtain that $N\leqslant \A_n$ if and only if $m$ and $l=\frac{q^m-1}{m(q-1)}$ are even. 
It follows from~\cite[Lemma~1.7(iii)]{Vasil'ev} that $l$ is even if and only if $q\equiv3\!\!\pmod4$.
Since $q\equiv3\!\!\pmod4$, the property (2) of Theorem~\ref{thm:Zassenhaus-1} is equivalent to $m\equiv2\!\!\pmod4$. 

Now $L=FN\leqslant\A_n$ and $N$ has the following representation: 
$$N=\langle a,b~|~a^{(q-1)l}=b^{(q-1)m}=1, a^{tl}=b^m, b^{-1}ab=a^q\rangle.$$ 
The group $N$ has three normal subgroups of index two: $S_1=\langle a, b^2\rangle$, $S_2=\langle a^2, b\rangle$, $S_3=\langle a^2, ab\rangle$. If $q^m=3^2$, then $N=Q_8$ has three normal subgroups of index~2, isomorphic to $\mathbb{Z}_4$. If $q^m>9$, then the group $S_1$ is not isomorphic to $S_2$ or $S_3$, however, $S_2\simeq S_3$. Hence, $L$ has two non-isomorphic subgroups $FS_1$ and $FS_2$ of index~2.

Now consider the exceptional cases when $p^r$ is one of integers $5^2, 7^2, 11^2,23^2, 29^2$, or $59^2$. 
According to~\cite[Theorem~2.4]{HeringG} the group $L=FN$ is a Frobenius group, where $F$~is the elementary abelian group of order $p^2$ and $N$ is one of the following groups:

(i) $p=5$ and $N\simeq\SL_2(3)$;

(ii) $p=7$ and $N\simeq 2^-\Ss_4\simeq\SL_2(3).2$;

(iii) $p=11$ and $N\simeq\SL_2(3)\times\mathbb{Z}_5$;

(iv) $p=11$ and $N\simeq\SL_2(5)$;

(v) $p=23$ and $N\simeq 2^-\Ss_4\times\mathbb{Z}_{11}\simeq\SL_2(3).2\times\mathbb{Z}_{11}$;

(vi) $p=29$ and $N\simeq\SL_2(5)\times\mathbb{Z}_7$;

(vii) $p=59$ and $N\simeq\SL_2(5)\times\mathbb{Z}_{29}$.

Note that the groups $\SL_2(3)$ and $\SL_2(5)$ have no subgroups of index~2. 
Thus, the only possibilities are (ii) and (v), where $N$ has one subgroup of index~2.
\end{proof}

\section{Equivalence of Rota---Baxter operators}

\begin{proposition}[\cite{Jiang}]\label{prop:RBviaSubgroups}
Let $B$ be a map defined on a finite group~$G$.
Construct $H_B = \{(B(g),gB(g))\mid g\in G\}\subset G\times G$.

(a) If $B$ is an RB-operator on $G$, then $H_B$ is a subgroup of $G\times G$.

(b) Let $H$ be a subgroup of $G\times G$.
Then there exists an RB-operator $B$ on $G$ such that $H = H_B$ if and only if 
$|H|=|G|$ and $\{ba^{-1}\mid (a,b)\in H\} = G$.    
\end{proposition}

Note that $\{ba^{-1}\mid (a,b)\in H\} = G$
is equivalent to the condition $H\cap \{(g,g)\mid g\in G\} = \{e\}$.

Given an element $x$ of a group $G$, the automorphism $g\mapsto g^x$ is denoted by $\alpha_x$.
Let us consider the following subgroup of~$\Aut(G\times G)$:
$$
Q(G) = \{ (\varphi,\varphi\,\alpha_x)\mid \varphi\in\Aut(G),\,x\in G\} 
\cup \{ \tau((\varphi,\varphi\,\alpha_x))\mid \varphi\in\Aut(G),\,x\in G \},
$$
where 
$\tau((\varphi,\varphi\,\alpha_x))\colon (g,h)\to ( \varphi\,\alpha_x(h), \varphi(g)  )$.

\begin{proposition} \label{prop:equivalentRB}
Let $B$ be an RB-operator on a finite group~$G$.
For every $\Phi\in Q(G)$, there exists an RB-operator $B'$ on $G$ such that $\Phi(H_B) = H_{B'}$.
\end{proposition}

\begin{proof}
It is clear that $|\Phi(H_B)| = |G|$, it remains to check that
$\{ba^{-1}\mid (a,b)\in \Phi(H_B)\} {=} G$.
Let us prove the statement for~$\Phi = (\varphi,\varphi\,\alpha_x)$, the proof for $\Phi = \tau((\varphi,\varphi\,\alpha_x))$ is analogous.
Thus, we have to find for every $g\in G$ such $(a,b)\in H_B$ that 
$\varphi(\alpha_x(b))\varphi(a)^{-1} = g$.
Equivalently, we search for $(a,b)\in H_B$ such that $bxa^{-1} = h$, where 
$h = x\varphi^{-1}(g)$ may be any element from $G$.
We may find $(a_0,b_0)\in H_B$ such that $b_0a_0^{-1} = x$.
Hence, we need to get a~pair $(a,b)\in H_B$ satisfying the equality
$$
bxa^{-1} = bb_0a_0^{-1}a^{-1} = bb_0(aa_0)^{-1} = h, 
$$
and it is possible, since $h = \tilde{b}\tilde{a}^{-1}$ for some $(\tilde{a},\tilde{b})\in H_B$.
Thus, $(a,b) = (\tilde{a},\tilde{b})(a_0,b_0)^{-1}\in H_B$.
\end{proof}

Let us call a pair of RB-operators $B,B'$ on a group $G$ {\bf equivalent}, if 
$\Phi(H_B) = H_{B'}$ for some $\Phi\in Q(G)$.
Let us interpret Proposition~\ref{prop:Btilde} in terms of equivalent RB-operators.
If $\Phi = \tau((\id,\id))$, we have $B' = \widetilde{B}$.
If $\Phi = (\varphi,\varphi)$, then $B' = B^{(\varphi)}$.

\begin{lemma}
Let~$G$ be a finite group and $B,B'$ be a pair of equivalent RB-operators on~$G$.
Then 

(a) $G^{(\circ)}(B)\simeq G^{(\circ)}(B')$,

(b) $B$ is trivial if and only if $B'$ is trivial,

(c) up to action of $B\to \widetilde{B}$, we have $\Imm(B)\simeq \Imm(B')$, $\Imm(\widetilde{B})\simeq \Imm(\widetilde{B'})$,

(d) $\Imm(B\widetilde{B})\simeq \Imm(B'\widetilde{B'})$,

(e) $B$ is splitting if and only if $B'$ is splitting,

(f) if $G = \ker(B)\cdot \Imm(B)$, then up to the action of $B\to \widetilde{B}$,
one has $G = \ker(B')\cdot \Imm(B')$.
\end{lemma}

\begin{proof}
(a) Define $\kappa\colon G\to H_B$ as follows,
$\kappa\colon g\to (B(g),gB(g))$.
Note that $\kappa$ is bijective and 
$\kappa(g\circ h)\to \kappa(g)\kappa(h)$,
thus, $G^{(\circ)}(B)\simeq H_B$, hence $G^{(\circ)}(B)\simeq H_B\simeq H_{B'}\simeq G^{(\circ)}(B')$.

(b), (c) It is trivial.

(d) Let us consider $\Phi = (\varphi,\alpha_x\varphi)\in Q(G)$, $\varphi\in\Aut(G)$, $x\in G$.
Note that $\Imm(B^{(\varphi)}\widetilde{B^{(\varphi)}})\simeq \Imm(B\widetilde{B})$.
Thus, it remains to study the case when~$B'$ comes from~$B$ with the help of~$\Phi = (\id,\alpha_x)$. Then 
$$
\Imm(\widetilde{B'}B')
= \alpha_x( \Imm(\widetilde{B}B) ),
$$
i.\,e. the required subgroups are isomorphic.

(e) It follows from (d) and Proposition~\ref{prop:SplittingCond}.

(f) It follows from the fact that $G = H\cdot L$ is an exact factorization if and only if $G = H\cdot \alpha_x(L)$ is an exact factorization for any $x\in G$.
\end{proof}

\begin{remark}
Let us show that one may not generalize Proposition~\ref{prop:equivalentRB} on the whole~$\Aut(G\times G)$.
For this, take $G = \A_4$ and identify $\A_4 \times \A_4$ with the corresponding subgroup in~$\Ss_8$. Then the map~$B$ defined on~$\A_4 = V\rtimes L$ by Proposition~\ref{prop:semidirect} for 
$$
V = \langle (12)(34),(13)(24)\rangle, \quad
L = \langle (234)\rangle
$$ 
and $B|_L = \id$ is a Rota---Baxter operator.
We have $H_B = \langle (234)(687), (13)(24), (12)(34)\rangle$.
Let us consider $\psi = \conj_{(23)}\in\Aut(\Ss_8)$, which induces the automorphism of~$G\times G$. Then $\psi(H_B) = \langle (243)(687), (13)(24), (12)(34)\rangle$
and now, $\psi(H_B)$ has a non-trivial intersection with the diagonal subgroup of $G\times G$. Thus, the condition (b) from Proposition~\ref{prop:RBviaSubgroups} does not hold. Note that $\psi$ has the form $(\varphi,\id)$ for the outer automorphism~$\varphi$ of~$\A_4$.
\end{remark}

\section{Properties of some explicit construction}

Perhaps, the next result may be derived from~\cite{Das}.
For the sake of completeness, we provide its proof.

\begin{lemma}\label{lem:Main}
Let $B$ be an RB-operator on a finite group $G$ such that 

(i) $R = \Imm(B)\cap \Imm(\widetilde{B})$ is abelian,

(ii) $G = \ker(B)\cdot \Imm(B)$.

Then $\widetilde{B}|_{\Imm(B)}$ is a~homomorphism onto abelian subgroup~$R$ of $\Imm(B)$,
and $B$ is defined on $G$ by Corollary~\ref{coro:Main}.
\end{lemma}

\begin{proof}
By Proposition~\ref{Prop:STSh}, $\ker(\widetilde{B})\unlhd\Imm(B)$ and we may consider the group $\Imm(B)/\ker(\widetilde{B})$ and $B,\widetilde{B}$ are RB-operators on it. 
Moreover, $\widetilde{B}$ is a homomorphism from $(\Imm(B))^{(\circ)}$ into~$R = \Imm(B\widetilde{B}) = \Imm(\widetilde{B}B)$, here we define $\circ$ via the RB-operator~$\widetilde{B}$:
$$
y\circ x = y\widetilde{B}(y)x\widetilde{B}(y)^{-1}.
$$
By homomorphism theorem, $\widetilde{B}$ is an isomorphism from $(\Imm(B))^{(\circ)}/\ker(\widetilde{B})$ onto~$R$. 
Since $\ker(\widetilde{B})\circ x = \ker(\widetilde{B})\cdot x$ for all $x\in G$,
we have $\Imm(B)^{(\circ)}/\ker(\widetilde{B}) = (\Imm(B)/\ker(\widetilde{B}))^{(\circ)}\simeq R$ and thus $(\Imm(B)/\ker(\widetilde{B}))^{(\circ)}$ is abelian.
Now, we consider $\widetilde{B}\colon (\Imm(B)/\ker(\widetilde{B}))^{(\circ)}\to \Imm(B)/\ker(\widetilde{B})$, since $\widetilde{B}$ is an isomorphism,
the group $\Imm(B)/\ker(\widetilde{B})$ is abelian too.
Since $\widetilde{B}|_{\Imm(B)}$ is an RB-operator on an abelian group,
we derive that $\widetilde{B}$ is an automorphism of $\Imm(B)/\ker(\widetilde{B})$.

Choose an independent set $r_1,\ldots,r_n$ of generators of an abelian $R$
and find $t_i\in \Imm(B)$ such that $\widetilde{B}(t_i) = r_i$, $i=1,\ldots,n$.
Since $\widetilde{B}$ is a homomorphism of $\Imm(B)/\ker(\widetilde{B})$,
$\widetilde{B}(t_i t_j) = k_{ij} r_i r_j$ for some $k_{ij}\in \ker(\widetilde{B})$.
By~\eqref{RB}, we write down
$$
r_i r_j
= \widetilde{B}(t_i)\widetilde{B}(t_j) 
= \widetilde{B}(t_i r_i t_j r_i^{-1} ).
$$
Since $\widetilde{B}$ is well-defined on the abelian group~$\Imm(B)/\ker(\widetilde{B})$
and $\ker(\widetilde{B})t_i t_j = \ker(\widetilde{B})t_i r_i t_j r_i^{-1}$,
we obtain $k_{ij} = e$ and $\widetilde{B}(t_i t_j) = r_i r_j$.
Similarly, we show that $\widetilde{B}(t_i t_j t_k) = r_i r_j r_k$ and so on.
Thus, $\widetilde{B}$ is a homomorphism from $\Imm(B)$ to $R$. 

Finally, $C = B|_{\Imm(B)}$ is an RB-operator on $\Imm(B)$ with $\Imm(\widetilde{C}) = \Imm(B\widetilde{B})$ normalizing $\ker(B)$ in $G$, since $\ker(B)$ is normal in the whole $\Imm(\widetilde{B})$.
\end{proof}

Now, we provide the conditions under which we may apply Lemma~\ref{lem:Main}.
Let $B$ be an RB-operator on a finite group~$G$. 
If $\ker(B)=\Imm(\widetilde{B})$, then $B$~is a splitting RB-operator by Proposition~7. 
Consider the case $\ker(\widetilde{B})<\Imm(B)$ and $\ker(B)<\Imm(\widetilde{B})$.

Denote $H = \Imm(B)$, $L = \Imm(\widetilde{B})$, and $R=H\cap L$. 
Then
\begin{equation} \label{OrderIntersectionFormula}
|G|=|\Imm(\widetilde{B})\Imm(B)|=\frac{|\Imm(B)|\cdot|\Imm(\widetilde{B})|}{|R|}
= |\ker(B)|\cdot|\Imm(B)|=|\ker(\widetilde{B})|\cdot|\Imm(\widetilde{B})|,
\end{equation}
and
\begin{equation} \label{OrderOfIntersection}
|R|=|\Imm(B):\ker(\widetilde{B})|=|\Imm(\widetilde{B}):\ker(B)|>1.
\end{equation}

\begin{lemma}\label{lem:CondToMain}
Let $B$ be a non-splitting RB-operator on a~finite group~$G$, $R = \Imm(B)\cap\Imm(\widetilde{B})$. Then

(a) $R\nleqslant\ker(B)$ or $R\nleqslant\ker(\widetilde{B})$, 

(b) if $R\cdot\ker(B)=\Imm(\widetilde{B})$, then $G = \Imm(B)\cdot\ker(B)$ is an exact factorization.  
\end{lemma}

\begin{proof}
(a) If $R\leqslant\ker(B)$ and $R\leqslant\ker(\widetilde{B})$, then 
$R\leqslant\ker(B)\cap\ker(\widetilde{B}) = 1$, a contradiction. 

(b) Since $R\leqslant \Imm(B)$, we have 
$$
G = \Imm(B)\cdot\Imm(\widetilde{B})
= \Imm(B)\cdot R\cdot\ker(B)
= \Imm(B)\cdot\ker(B),
$$
where the last factorization is exact.
\end{proof}

\begin{corollary} \label{coro:CondToMain}
Let $B$ be a non-splitting RB-operator on a~finite group~$G$,
$R = \Imm(B)\cap\Imm(\widetilde{B})$ and $|R|$ is prime.
Then, up to action $B\to\widetilde{B}$, $B$ is defined on $G$ due to Lemma~\ref{lem:Main}.
\end{corollary}

\begin{proof}
By Lemma~\ref{lem:CondToMain}a), we may suppose that $R\nleqslant\ker(B)$.
Then $|R\cap \ker(B)| = 1$ and thus, $R\cdot\ker(B)=\Imm(\widetilde{B})$.
By Lemma~\ref{lem:CondToMain}b), we get $G = \Imm(B)\cdot\ker(B)$.
Therefore, all conditions of Lemma~\ref{lem:Main} are satisfied.
\end{proof}

\begin{remark}\label{rem:R=2}
Corollary~\ref{coro:CondToMain} has a simple form, when $|R| = 2$ and $R\not\leqslant \ker(B)$.
We have the exact factorization $G = \ker(B)\cdot \Imm(B)$.
Choose $t\in \Imm(B)\setminus \ker (\widetilde{B})$ and denote $r = \widetilde{B}(t)$.
Each $x\in G$ has the form $x = kl$, where $k \in \ker(B)$ and $l\in\Imm(B)$. Since $|\Imm(B):\ker(\widetilde{B})|=2$, $l=t^\delta y$ for some $y\in \ker(\widetilde{B})$ and $\delta\in\{0,1\}$. 
Therefore, $B$ acts on~$G$ as follows:
\begin{equation}\label{eq:ActionBWhen|R|=2}
B(x)
= B(kl)
= B(l)
= l^{-1}\widetilde{B}(l^{-1})
= l^{-1}\widetilde{B}(y^{-1}t^\delta)
= l^{-1}r^\delta.
\end{equation}
\end{remark}

\section{Rota---Baxter operators on dihedral groups}

Below, we define the dihedral group $D_{2n}$ as follows,
$$
D_{2n} = \langle r,s \mid r^n = s^2 = e,\,r^s = r^{-1}\rangle.
$$

\begin{theorem} \label{thm:Dihedral}
(a) There are no non-splitting RB-operators on $D_{2n}$, where $n$ is odd.

(b) Let $n$ be even. Then all non-splitting RB-operators~$B$ on $D_{2n}$, up to action $B\to\widetilde{B}$, come from an exact factorization $D_{2n} = HL$, where $H = \ker(B)$ and $L = \Imm(B)$.
Moreover, $\widetilde{B}|_L$ is a~homomorphism onto $\mathbb{Z}_2$ or $\mathbb{Z}_2^2$.
Hence, $B$ on $D_{2n}$ is defined by Corollary~\ref{coro:Main}.
\end{theorem}

\begin{proof}
(a) We want to show that $r^m s\in \ker(B\widetilde{B})$ for all $0\leqslant m<n$.
It would imply that $\ker(B\widetilde{B})$ contains at least $n+1$ elements;
since $\ker(B\widetilde{B})$ is a subgroup, we prove that $B\widetilde{B}(g) = e$ for all $g\in G$. By Proposition~\ref{prop:SplittingCond}, $B$ is splitting.
One of $B(r^m s),\widetilde{B}(r^m s)$ lies in $\langle r\rangle$, say $B(r^m s) = r^k$ for some~$k$. Therefore
\begin{equation} \label{r^2k-formula}
r^{2k} = B(r^m s)B(r^m s) 
= B(r^m sr^k r^m sr^{-k})
= B(r^{-2k}),
\end{equation}
and $r^{2k}\in \ker(\widetilde{B})$. 
Since $n$ is odd, we get that $r^k\in \ker(\widetilde{B})$, and we are done.

(b) Let $n$ be even. For $n=2$, one may check the statement directly.
Let $n>2$. Define $R = \Imm(B)\cap \Imm(\widetilde{B})$.
When $|R|=2$, the statement follows from Remark~\ref{rem:R=2}.
Hence, we may suppose that $|R|>2$.

Assume that $\Imm(B) = \langle r^{d_1}\rangle$ and 
$\Imm(\widetilde{B}) = \langle r^{d_2}, s\rangle$, where $d_1,d_2\,|\,n$. 
Since $\ker(B)\unlhd\Imm(\widetilde{B})\simeq D_{2n/d_2}$, we have
$\ker(B) = \langle r^{d_2d_3}\rangle$ for some~$d_3$.
Thus, $\Imm(\widetilde{B})/\ker(B)$ is dihedral and noncyclic, whereas $\Imm(B)/\ker(\widetilde{B})$ is cyclic, a~contradiction to~\eqref{FactorIso}.

It remains to study the situation when 
both $\Imm(B)$ and $\Imm(\widetilde{B})$ are dihedral groups.

{\sc Case 1}.
$|\Imm(B)|, |\Imm(\widetilde{B})|<|G|$.
Denote $H = \Imm(B)$ and $L = \Imm(\widetilde{B})$.
Applying Lemma~\ref{lem:ImagesManipul} and induction, we have
$$
|\Imm(B^2\widetilde{B})| 
= |\Imm(B|_H)\cap \Imm({\widetilde{B}}|_H)| = 2^a, \quad
|\Imm(B\widetilde{B}^2)| 
= |\Imm(B|_L)\cap \Imm({\widetilde{B}}|_L)| = 2^b
$$ 
for $a,b\in \{0,1,2\}$.
By Lemma~\ref{lem:ImagesManipul} and by~\eqref{ImageFactorization}, we have
$$
R = \Imm(B\widetilde{B})
= B(R)\cdot\widetilde{B}(R)
= \Imm(B^2\widetilde{B})\cdot\Imm(B\widetilde{B}^2).
$$
Hence, analogously to~\eqref{OrderIntersectionFormula}, we deduce that 
$|R| = 2^c$ for $0\leqslant c\leqslant 4$.

{\sc Case 1A}. Suppose that $R$ contains some reflection, so, up to the action of $\Aut(D_{2n})$, we may assume that 
$\Imm(B) = \langle r^{d_1},s\rangle$ and 
$\Imm(\widetilde{B}) = \langle r^{d_2}, s\rangle$, where $d_1,d_2\,|\,n$.

According to~\eqref{ImageFactorization}, we have $(d_1,d_2) = 1$ and so, $n = d_1d_2d_3$ for suitable~$d_3$.
Since 
$$
|R|\cdot |G|
= |\Imm(B)|\cdot|\Imm(\widetilde{B})| 
= \frac{2n}{d_1}\cdot\frac{2n}{d_2}
= 2d_3\cdot 2n, 
$$
we derive that $|R| = 2d_3$.
Hence, $\ker(\widetilde{B})$ is a normal subgroup of index~$2d_3$ in $\Imm(B)$, i.\,e. $\ker(\widetilde{B}) = \langle r^{d_1d_3}\rangle$.
Analogously, $\ker(B) = \langle r^{d_2d_3}\rangle$.

Since we suppose that $|R|>2$, it remains to study the case $|R| = 2^c$ for $2\leqslant c\leqslant 4$. Hence, $d_3 = 2^{c-1}>1$.
Thus, $(d_1,d_3) = 1$ or $(d_2,d_3) = 1$.
It means that we have an exact factorization, say,
$G = \ker(B)\cdot\Imm(B)$. 
Such factorization implies that $\Imm(B^2) = \Imm(B)$.
By induction,
\begin{equation}\label{eq:dihedral-induction}
|R| = |(\widetilde{B}B)(G)| 
= |(\widetilde{B}B^2)(G)| 
= |(\widetilde{B}B)(H)| 
= |(B\widetilde{B})(H)| 
\leqslant 4,
\end{equation}
and we have exactly the case $R\simeq \mathbb{Z}_2\times \mathbb{Z}_2$.
Hence, the statement follows from Lemma~\ref{lem:Main}.

{\sc Case 1B}. Now, consider the case, when $R$ does not contain any reflection.
Denote $\Imm(B) = \langle r^{d_1},s\rangle$ and 
$\Imm(\widetilde{B}) = \langle r^{d_2}, r^k s\rangle$, where $d_1,d_2\,|\,n$ and $d_2\not|\,k$.
Again, by~\eqref{ImageFactorization}, we have either $(d_1,d_2) = 1$ or $(d_1,d_2) = 2$. In the first case, $R = \Imm(B)\cap \Imm(\widetilde{B}) = 1$, a contradiction.
In the second one, we get $R = \{e,r^{n/2}\}$, a contradiction with the assumption $|R|>2$.

{\sc Case 2}.
Only one of $|\Imm(B)|, |\Imm(\widetilde{B})|$ is equal to $|G|$.
Suppose that $\Imm (B) = G$. 
Choose $a\in G$ such that $B(a) = r$. Denote $x = a^{\circ(n/2)}$, then $B(x) = r^{n/2}$. 
By Proposition~\ref{prop:CircPower}, we may rewrite $x$ as 
$$
x = (B_+(a))^{n/2}(B(a))^{-n/2}
= (ar)^{n/2}r^{n/2}.
$$

We have either $a = r^m s$, $0\leqslant m<n$ or $a = r^m$, $1\leqslant m<n$. Assume that $a = r^m s$ and 
$n/2$ is odd. Then we have $x = r^{n/2+m-1}s$. 
By Lemma~\ref{lem:centerPreimage}, $x$ is centralized by $\Imm(\widetilde{B})$.
Hence, 
$\Imm(\widetilde{B})\leqslant C_{D_{2n}}(x) = \{e,r^{m-1} s,r^{n/2},r^{n/2+m-1}s\}\simeq \mathbb{Z}_2\times \mathbb{Z}_2$.
Moreover, we have the exact factorization $G = \ker(B)\cdot\Imm(B)$,
thus, we may again apply Lemma~\ref{lem:Main}.

If $a = r^m s$ and $n/2$ is even or $a = r^m$, then we get that $x=r^{n/2}$.
Define $K = \langle r^{n/2}\rangle$, the normal subgroup of order~2 in $D_{2n}$.
Since $B(r^{n/2}) = r^{n/2}$, we may consider induced RB-operators $B,\widetilde{B}$ on $G/K$.
Note that $\widetilde{B}(r^{n/2}) = e$.
By induction, $\Imm(\widetilde{B}|_{G/K})$ is a~subgroup of $\mathbb{Z}_2\times \mathbb{Z}_2$.
If $r^{n/2}\not\in \Imm(\widetilde{B})$, 
then $\Imm(\widetilde{B})\leqslant \mathbb{Z}_2\times \mathbb{Z}_2$,
and the statement holds by Lemma~\ref{lem:Main} due to the exact factorization $G = \ker(B)\cdot\Imm(B)$.
If $r^{n/2}\in \Imm(\widetilde{B})$,
then $\Imm(\widetilde{B})\leqslant \mathbb{Z}_2\times \mathbb{Z}_2\times \mathbb{Z}_2$ (since $r^{n/2}$ is central, we get an additional factor~$\mathbb{Z}_2$).
The case $\Imm(\widetilde{B})\simeq \mathbb{Z}_2^3$ 
is impossible, because $\mathbb{Z}_2^3$ may not be a subgroup of any dihedral group. Hence, we again apply Lemma~\ref{lem:Main}.

{\sc Case 3}.
$|\Imm(B)|=|\Imm(\widetilde{B})|=|G|$.
Recall that the map $B_+$, defined as $B_+(g)=gB(g)$, is a homomorphism from $G^{(\circ)}$ in $G$. Since $|\Imm(\widetilde{B})|=|\Imm(B_+)|=|G|$, we have $\ker(B_+)=e$ and $B_+$ is an isomorphism. Therefore, $\varphi = B^{-1}B_+$ is an automorphism of~$G^{(\circ)}\simeq D_{2n}$. It is trivial to check that $\varphi$ is fixed-point free. Since $n\geqslant 3$, the center of $D_{2n}$ is preserved by every automorphism. Hence, $D_{2n}$ does not admit fixed-point free automorphisms, and we arrive at a~contradiction.
\end{proof}

In~\cite{Das}, the correspondence between extensions of RB-groups and cohomologies was studied. 
As examples, the authors considered RB-operators on $\Ss_3$ and $D_8$.

\begin{corollary} \label{coro:S3}
All RB-operators on $\Ss_3 \simeq D_6$ are splitting.    
\end{corollary}

The next example appeared in~\cite[Example\,1.13]{SV} in the context of skew left braces, we reformulate it in terms of RB-operators.

\begin{example}
Define a map $B\colon \Ss_3\to \Ss_3$ as follows,
$$
B\colon (ij) \to (23), \quad
B\colon (123)^k \to e, \quad i\neq j,\ 1\leqslant i,j,k \leqslant 3.
$$
Then $B$ is a Rota---Baxter operator on $\Ss_3$.
By Corollary~\ref{coro:S3}, we know that $B$ is splitting, indeed,
$\Ss_3 = \langle (123)\rangle \cdot \langle (23)\rangle = \ker (B)\cdot \ker(\widetilde{B})$. Thus, $\Ss_3^{(\circ)}\simeq \mathbb{Z}_6$.
\end{example}

\begin{example}
In~\cite[Example\,1.18]{SV}, we have a Rota---Baxter operator~$B$ on $D_8$ defined as follows,
$$
B(r^k) = e, \quad
B(r^ks) = r, \quad 0\leqslant k<4.
$$
Note that $\widetilde{B}$ is invertible and $D_8^{(\circ)}\simeq Q_8$.
\end{example}

The next example shows that $\widetilde{B}B \neq B\widetilde{B}$ in general.

\begin{example}\label{exm:BBTildeNotCommute}
Let us consider the exact factorization $D_{16} = HL$ for
$H = \{e,s\}$, $L = \{e,r^2,r^4,r^6,rs,r^3s,r^5s,r^7s\}$.
We define $B$ on $D_{16}$ in such a manner that 
$$H = \ker(\widetilde{B}), \quad
\ker(B) = \langle r^4,rs\rangle,\text{ and } B(r^2) = B(r^6) = B(r^3s) = B(r^7s) = r^4.$$
Hence, $B$ is a~homomorphism from $L$ onto $\langle r^4\rangle\simeq \mathbb{Z}_2$,
moreover, $\Imm(B)|_L$ normalizes $\ker(\widetilde{B})$.
Thus, we get an RB-operator on~$D_{16}$.
Note that 
\begin{gather*}
(B\widetilde{B})(r) 
= B(\widetilde{B}(r)) 
= B(\widetilde{B}(sr)))
= B(\widetilde{B}(r^7s)))
= B(r^7sB(r^7 s)) 
= B(r^3s)
= r^4, \\
(\widetilde{B}B)(r)
= \widetilde{B}(B(r))
{=} \widetilde{B}( r^7\widetilde{B}(r^7))
= \widetilde{B}( r^7\widetilde{B}(sr^7))
{=} \widetilde{B}( r^7\widetilde{B}(rs))
= \widetilde{B}( r^7\cdot rs)
{=} \widetilde{B}(s) = e.
\end{gather*}
\end{example}

The following example shows that the center of a~group may not be preserved by an RB-operator.

\begin{example}\label{exm:GeneralInvertible}
Define $B$ on $D_{2n}$ in such a way that
$$
\widetilde{B}(r^{2k}) = e, \quad
\widetilde{B}(r^{2k+1}) = r^{n/2}s, \quad
\widetilde{B}(r^{2k}s) = r^{n/2}, \quad
\widetilde{B}(r^{2k+1}s) = s, \quad 0\leqslant k<n/2.
$$
Then $\widetilde{B}$ is a homomorphism from~$D_{2n}$ onto~$\mathbb{Z}_2\times \mathbb{Z}_2$
and $B$ is a Rota---Baxter operator on~$D_{2n}$ due to the trivial factorization 
$D_{2n} = \ker(B)\cdot \Imm(B) = 1\cdot D_{2n}$.
This example shows that the center of a~group may not be preserved by an RB-operator.
If $4\not|\,n$, we have 
$B(r^{n/2}) = r^{n/2}\widetilde{B}(r^{n/2}) = s$.
\end{example}

\begin{example}
Consider the exact factorization $G = D_{60} = HL = \langle r^{10}\rangle\cdot \langle r^3,s\rangle$.
We define an RB-operator on $L\simeq D_{20}$ as in Example~\ref{exm:GeneralInvertible}.
Further, we extend such~$B$ on the whole~$G$, putting $\ker(B) = H$.
Surely, $H$ is normalized by~$\Imm(\widetilde{B})$, since $H$ is normal in~$G$.
\end{example}

\begin{remark}
The descendent group~$D_{2n}^{(\circ)}$ for odd prime~$n$
is discussed in~\cite[Corollary\,6.5]{Byott}.
\end{remark}

Recall that a generalized quaternion group~$Q_{4n}$ of order~$4n$ is defined as follows:
$$
Q_{4n} = \langle r,s \mid r^{2n} = e,\,s^2 = r^n,\,r^s = r^{-1}\rangle.
$$

\begin{corollary}
Let $B$ be a non-splitting RB-operator on~$Q_{4n}$, where $n$ is odd.
Then $B$ is defined by~Remark~\ref{rem:R=2}.
\end{corollary}

\begin{proof}
It is known that $Z(Q_{4n}) = \{e,r^n\}$.
Since $r^n$ is a unique element of order~2 in~$Q_{4n}$, we conclude by Lemma~\ref{lem:centerPreimage} that the normal subgroup $K = \langle r^n\rangle$ is $B$-invariant: either $B(r^n) = r^n$ or $B(r^n) = e$.
Therefore, $B$ is an RB-operator on~$Q_{4n}/K \simeq D_{2n}$.
By Theorem~\ref{thm:Dihedral}(a), we know that all RB-operators on~$D_{2n}$ are splitting. Hence, $|\Imm(B)\cap\Imm(\widetilde{B})|\leqslant 2$. Since $B$ is non-splitting, we have $|R|=2$.
\end{proof}

Let us show that the case $\Imm(B)\cap\Imm(\widetilde{B})\simeq \mathbb{Z}_2$ is realizable on~$Q_{4n}$.

\begin{example}
Consider the exact factorization~$Q_{60} = \langle r^{10}\rangle\cdot \langle r^3,s\rangle = H\cdot L$
and define $B\colon Q_{60}\to Q_{60}$ as follows,
$\ker (B) = H$ and $\widetilde{B}\colon L\to L$ acts by the formula
$$
\widetilde{B}( r^{3k} s^i ) = \begin{cases}
r^{15}, & i = 1, \\
e, & i = 0.
\end{cases}
$$
It is an RB-operator on~$Q_{60}$, since $\widetilde{B}$ is an endomorphism of~$L$ and $\Imm(\widetilde{B}|_L)$ normalizes~$H$. 
\end{example}

\section{Rota---Baxter operators on alternating groups}

This section is devoted to the study of Rota---Baxter operators on simple groups~$\A_n$, $n\geqslant5$.
For completeness, let us deal with $\A_n$ for small $n$.
Since $\A_3 \simeq \mathbb{Z}_3$ is abelian, an RB-operator on it is nothing more than a homomorphism.

Computations show that, up to equivalence, we have the following non-trivial RB-operators on $\A_4$:

(B1) splitting one coming from the exact factorization $\A_4 = HL$,
where $H=\langle(234)\rangle \simeq \mathbb{Z}_3$ and $L = V_4$ is the Klein four-group.

(B2) non-splitting one coming from the exact factorization $\A_4 = HL$,
where $H = \{e\}$, $L = \A_4$, and an RB-operator~$C$ on $L$ is a natural homomorphism from 
$L = \mathbb{Z}_3 \rtimes V_4$ onto $\mathbb{Z}_3$.
Thus, $\Imm(B\widetilde{B}) \simeq \mathbb{Z}_3$ and Lemma~\ref{lem:Main} is again valid.

The fact that $\A_4^{(\circ)}$ is isomorphic either to $\A_4$ (for trivial RB-operators) or to $\mathbb{Z}_6\times \mathbb{Z}_2$ may be derived from~\cite{Crespo}.

The study of splitting Rota---Baxter operators is the same as the study of exact factorizations, which are described in~\cite{WiegoldW}. Therefore, we will focus on non-splitting RB-operators on $\A_n$.

Below, we will apply the following conditions on numbers $m,q$:

(i) $q$ is a prime-power;

(ii) Each prime divisor of $m$ divides $q-1$;

(iii) $q\equiv3\!\!\pmod4$ and $m\equiv2\!\!\pmod4$.

\begin{theorem}\label{thm:alternating}
Let $B$ be a Rota---Baxter operator on a simple alternating group $\A_n$. \\ 
(1) $B$ is non-splitting if and only if one of the following holds:
\begin{itemize}
\item[$(a)$] $n=q^m$ and $m,q$ satisfy the conditions (i)--(iii);
\item[$(b)$] $n=q^m+1$ and $m,q$ satisfy the conditions (i)--(iii). 
\end{itemize}
(2) Let $B$ be a non-splitting Rota---Baxter operator on $\A_n$. Then there are $s$ non-equivalent Rota---Baxter operators on $\A_n$, where
\begin{itemize}
\item[$(a)$] $s=3$ for $n\in\{7^2, 23^2\}$;
\item[$(b)$] $s=2$ for $n=q^m$, $n\neq9$; 
\item[$(c)$] $s=1$ for $n\in\{9, q^m+1\}$. 
\end{itemize}
\end{theorem}

\begin{proof}
We assume that $B$ is a non-splitting RB-operator on $G=\A_n$. According to Proposition~\ref{Prop:STSh}, we have
$$
G=\Imm(B)\cdot\Imm(\widetilde{B})=H\cdot L,
$$
$$
\ker(\widetilde{B})\unlhd\Imm(B),\quad \ker(B)\unlhd\Imm(\widetilde{B}).
$$
If $\ker(B)=1$ then $B$ is a splitting RB-operator~\cite[Theorem~50]{BG}.
If $\ker(B)=\Imm(\widetilde{B})$ then $B$~is also a splitting RB-operator by Proposition~\ref{prop:SplittingCond}. 
The same is true for $\ker(\widetilde{B})$. 
Therefore, we consider only the following case: 
$$
1<\ker(\widetilde{B})<\Imm(B), \qquad 1<\ker(B)<\Imm(\widetilde{B}).
$$

Define $R=H\cap L$. By Lemma~\ref{lem:CondToMain}(a) it follows that either 
$R\nleqslant\ker(B)$ or $R\nleqslant\ker(\widetilde{B})$. 

Since $\ker(\widetilde{B})>1$ and $\ker(B)>1$, we have
$G=HL$, where $H$ and $L$ are proper subgroups of $G$. According to \cite[Theorem~D]{LPS_fact}, 
$$L \text{ is } k\text{-homogeneous } (1\leqslant k\leqslant 5),\quad \A_{n-k}\unlhd H\leqslant (\Ss_{n-k}\times \Ss_k)\cap\A_n.$$

We consider all the possibilities for $k$.
In order to avoid non-simple groups $\A_n$, we assume that $n\geqslant10$. The cases $5\leqslant n\leqslant 9$ can be checked manually using the information from~\cite{Atlas} or with the help of GAP.
\vspace{0.3em}

$\boldsymbol{k=1.}$ In this case, $H=\A_{n-1}$ is a simple group and it follows from the proof of~\cite[Theorem~53]{BG} that $B$ must be splitting.
\vspace{0.3em}

$\boldsymbol{k=5.}$ By~\cite[Theorem~1]{Kantor} there is no 5-homogeneous but not 5-transitive finite group. Thus, $L$ is 5-transitive, and by~\cite[Theorem~5.3]{Cameron81} $L$ must be one of the Mathieu groups $M_{12}$ or $M_{24}$. Mathieu groups are simple, and we do not have a non-splitting RB-operator.\vspace{0.3em}

$\boldsymbol{k=4.}$ The Mathieu groups $M_{11}$ and $M_{23}$ are the only finite 4-transitive groups, which are not 5-transitive and so have not been considered. Since these two groups are simple, we need to consider 4-homogeneous but not 4-transitive groups. By~\cite[Theorem~1]{Kantor} we have $L\in\{\PSL_2(8),\operatorname{P\Gamma L}_2(8),\operatorname{P\Gamma L}_2(32)\}$.

Since $n\geqslant10$, we need to consider only the case $L = \operatorname{P\Gamma L}_2(32)$.
By~\cite{Kantor}, $L=\operatorname{P\Gamma L}_2(32)$ acts sharply 3-transitive on $\Omega=\{1,2,\ldots,33\}$. Since the only proper normal subgroup of $L$ is $\PSL_2(32)$, we have $\ker(B)=\PSL_2(32)$ and $|\operatorname{P\Gamma L}_2(32):\PSL_2(32)|=5$. 
Then $R\simeq \mathbb{Z}_5$ and $R\cdot\ker(B)=\Imm(\widetilde{B})$, since $(|R|,|\ker(B)|)=1$. 
Therefore 
$G=\Imm(B)\cdot\ker(B)=H\cdot\PSL_2(32)$ 
is the exact factorization. 
Since $\PSL_2(32)$ acts sharply 3-transitive on~$\Omega$, we get that $H=\A_{30}$~\cite[Theorem A(I)]{WiegoldW}. Thus, $H$ is simple, and we do not have a~non-splitting RB-operator.
\vspace{0.3em}

$\boldsymbol{k=2.}$ In this case, $\A_{n-2}\unlhd H\leqslant (\Ss_{n-2}\times \Ss_2)\cap\A_n$. Since $H$ is not simple, we have 
$\Imm(B)=H=(\Ss_{n-2}\times \Ss_2)\cap\A_n\simeq\Ss_{n-2}$. 
The only proper normal subgroup of $\Ss_{n-2}$ is $\A_{n-2}$. Therefore 
$\ker(\widetilde{B})=\A_{n-2}$ and $|R|=|\Ss_{n-2}:\A_{n-2}|=2$. 

\noindent (a) Assume that $R\nleqslant\ker(B)$. By Lemma~\ref{lem:CondToMain}(b) we have the exact factorization
$$
G=\Imm(B)\ker(B)=\Ss_{n-2}\cdot K.
$$
According to~\cite[Theorem A(II)]{WiegoldW} $K=\operatorname{ASL}(1,q)$, where $q\equiv3\!\pmod4$. Since $\ker(B)\unlhd\Imm(\widetilde{B})$ and $|\Imm(\widetilde{B}):\ker(B)|=|R|=2$, it follows that $\Imm(\widetilde{B})=\operatorname{AGL}(1,q)$. 
However, $\operatorname{AGL}(1,q)$ contains odd permutations; a contradiction. 

\noindent (b) Assume that $R\nleqslant\ker(\widetilde{B})$. 
By Lemma~\ref{lem:CondToMain}(b) we have the exact factorization
$$
G=\Imm(\widetilde{B})\ker(\widetilde{B})=L\cdot\A_{n-2}.
$$
According to~\cite[Theorem A(I)]{WiegoldW} $L$ is a sharply 2-transitive group.
Since $|\Imm(\widetilde{B}):\ker(B)|=2$, $L$~must contain a~normal subgroup of index~2 and we may apply Corollary~\ref{coro:Zassenhaus}. 

At first, consider the case (1) of Corollary~\ref{coro:Zassenhaus}, where $L=L(m,q,t)=FN$. Then we have the following conditions:

(1) Each prime divisor of $m$ divides $q-1$;

(2) If 4 divides $m$, then 4 divides $q-1$;

(3) $q\equiv3\!\!\pmod4$.

Since $q\equiv3\!\!\pmod4$, then $(2)$ is equivalent to the condition $m\equiv2\!\!\pmod4$ and we have the conditions (i)--(iii) of the theorem. By Corollary~\ref{coro:Zassenhaus}(1), $L=FN$ has two non-isomorphic normal subgroups $FS_1$ and $FS_2$.

Thus, we apply Remark~\ref{rem:R=2} and get two non-equivalent RB-operators on $G$ (see examples~\ref{exm:A9} and~\ref{exm:An} below). 

The cases (2) and (3) of Corollary~\ref{coro:Zassenhaus} provide two other sharply 2-transitive groups $L=FN$ of degree $n=7^2$ and $n=23^2$, respectively. Each of these groups has only one normal subgroup of index~2, 
and again by Remark~\ref{rem:R=2} we get one RB-operator in each case. 
Note that if $n=7^2$ we have two non-isomorphic sharply 2-transitive groups $L_1=L(m,q,t)=L(2,7,1)$ and $L_2=FN$ with $N\simeq\SL_2(3).2$. As it was mentioned above, $L_1$ provides two non-equivalent RB-operators on $G$ and $L_2$ provides the third one. Since $L=\Imm(\widetilde{B})$ and $L_1\not\simeq L_2$, we have three non-equivalent RB-operators on $G=\A_{7^2}$. Similarly, we get three non-equivalent RB-operators on $G=\A_{23^2}$.

$\boldsymbol{k=3.}$ In this case, $\A_{n-3}\unlhd H\leqslant (\Ss_{n-3}\times \Ss_3)\cap\A_n$ and $L$ is 3-homogeneous. 
Since $H$~is not simple, we have 
$$H\in\{(\Ss_{n-3}\times\Ss_3)\cap\A_n, (\Ss_{n-3}\times\Ss_2)\cap\A_n, \A_{n-3}\times\A_3\}.$$

\noindent (a) At first, consider the case 
$H=(\Ss_{n-3}\times\Ss_2)\cap\A_n\simeq\Ss_{n-3}.$ The only proper normal subgroup of $H$ is $\A_{n-3}$. Therefore
$\ker(\widetilde{B})=\A_{n-3}, |R|=|H:\ker(\widetilde{B})|=2$, and $$|\Imm(\widetilde{B})|=|L|=\frac{|G|\cdot|R|}{|H|}=\frac{(n!/2)\cdot2}{(n-3)!}=n(n-1)(n-2).$$
Then $L$ is a sharply 3-transitive group. Such groups have degree $q+1=p^r+1$ for some prime $p$. If $p=2$ then by~\cite[Theorem~2.1]{HuppertB} $L\simeq\PSL_2(q)$ is simple, which is not the case.

Let $p$ be an odd prime. By~\cite[Theorem~2.6]{HuppertB}, we have either $L=\PGL_2(q)$, or $r$~is even and $L=M(p^{r})\simeq\PSL_2(q).2$. In both cases, $\PSL_2(q)$ is a subgroup of index~2 in~$L$. Since $\PGL_2(q)$ contains an element $g$ of order $q+1$, we have $g\notin\A_{q+1}=\A_n$ and $L\neq\PGL_2(q)$.

The group $L=M(p^{r})$ acts on the projective line $\Omega=GF(p^r)\cup\{\infty\}$ (see,~\cite[Example~1.3]{HuppertB}). 
The stabilizer $L_\infty$ is a sharply 2-transitive Frobenius group of odd degree~$q$. This case was considered above, and we have the conditions (i)--(iii) on the numbers~$m$,~$q$.

Under these conditions, we have the exact factorization $G=\A_{n-3}\cdot L$, where $\ker(\widetilde{B})=\A_{n-3}\unlhd H=(\Ss_{n-3}\times \Ss_2)\cap\A_n$ and $\ker(B)=\PSL_2(q)\unlhd L=M(q^m)$. Now we can apply Remark~\ref{rem:R=2} to get one RB-operator in each case.

\vspace{0.5em}
\noindent (b) $H=(\Ss_{n-3}\times\Ss_3)\cap\A_n\simeq (\A_{n-3}\times\A_3):\mathbb{Z}_2$. 
We have
$$
|R| = \frac{|H|\cdot|L|}{|G|}
= \frac{(n-3)!\cdot3\cdot|L|}{(n!/2)}
= 3\cdot\frac{2|L|}{(n-2)(n-1)n}.
$$

If $L$ is 3-transitive, then $n(n-1)(n-2)$ divides $L$. If $L$ is not 3-transitive, then $|L|$ is divisible by $|\PSL_2(q)|=(q-1)q(q+1)/2=(n-2)(n-1)n/2$ or $L=\operatorname{A\Gamma L}_1(32)$ (see~\cite[Theorem~1(ii),(iii)]{Kantor}). 

Since $|\operatorname{A\Gamma L}_1(32)|=32\cdot31\cdot5$, we have $|R|=30/33$ and this case is impossible.
In two other cases we get that 3~divides $|R|=|\Imm(B):\ker(\widetilde{B})|$. 

On the other hand, $\Imm(B)=H=(\A_{n-3}\times\A_3)\!:\!\mathbb{Z}_2$ has only three proper normal subgroups: $\A_{n-3}\times\A_3, \A_{n-3}$, and $\A_3$. 

Since $|H:(\A_{n-3}\times\A_3)|=2$, the case $\ker(\widetilde{B})=\A_{n-3}\times\A_3$ is impossible.

If $\ker(\widetilde{B})=\A_3$ then $|R|=|\Imm(B)|/3=(n-3)!$, a contradiction with $R\leqslant L$.

Assume that $\ker(\widetilde{B})=\A_{n-3}$. Then 
$$
|R| = |\Imm(B):\ker(\widetilde{B})| = 6, \quad 
|L| = \frac{|G|\cdot|R|}{|H|}
= \frac{(n!/2)\cdot 6}{((n-3)!/2)\cdot6}
= n(n-1)(n-2).
$$

Then $L$ is a sharply 3-transitive group. According to case~(a) we have $L\simeq\PSL_2(q).2$. Therefore, $L$ has proper normal subgroups only of index~2. 
It is a contradiction with $|R|=|\Imm(\widetilde{B}):\ker(B)|=|L:\ker(B)|=6$.

\vspace{0.5em}
\noindent (c) 
$H=\A_{n-3}\times\A_3.$ The only proper normal subgroups of $H$ are $\A_{n-3}$ and $\A_3$. If $\ker(\widetilde{B})=\A_3$ then $|R|=\frac{|H|}{|\A_3|}=\frac{(n-3)!}{2}$, a contradiction with $R\leqslant L$. 

Therefore, $\ker(\widetilde{B})=\A_{n-3}$ and $|R|=|H:\ker(\widetilde{B})|=3$. 
We have
$$
|\Imm(\widetilde{B})| 
= |L| 
= \frac{|G|\cdot|R|}{|H|}
= \frac{(n!/2)\cdot3}{((n-3)!/2)\cdot3}
= n(n-1)(n-2).
$$
Thus, $L$~is a~sharply $3$-transitive group. According to case~(a) we have $L\simeq\PSL_2(q).2$. 
Therefore, $L$~has proper normal subgroups only of index~2. 
It is a contradiction with $|R|=|\Imm(\widetilde{B}):\ker(B)|=|L:\ker(B)|=3$.
\end{proof}

\begin{example}\label{exm:A9}
We start with an exact factorization $\A_9 = \ker(B)\cdot\Imm(B)=\A_{7}\cdot L$, where
$$
\A_7 = \langle (1234567), (123) \rangle,\quad
L = \langle (1456)(2789), (1683)(2579) \rangle \simeq (\mathbb{Z}_3\times\mathbb{Z}_3):Q_8.
$$
Then $H=\Imm(\widetilde{B})=(\Ss_7\times\Ss_2)\cap\A_9=\langle \A_7, r \rangle$, where
$r=(17)(25)(34)(89)$. Note that $R=H\cap L=\langle r\rangle$.
The group~$L$ has three isomorphic subgroups of index~2, and we choose one of them,
$$
\ker(\widetilde{B})=S=\langle (167)(238)(459), (1683)(2579) \rangle \simeq (\mathbb{Z}_3\times\mathbb{Z}_3):\mathbb{Z}_4.
$$
Then the RB-operator~$B$ on $\A_9$ is defined by Remark~\ref{rem:R=2}.
Namely, each $x\in\A_9$ has the form $x = kl = kt^\delta y$, where $k \in \ker(B)$, $y\in \ker(\widetilde{B})$ and $\delta\in\{0,1\}$. The operator~$B$ acts as follows:
$$
B(x)=B(kt^\delta y)=l^{-1}r^\delta.
$$
\end{example}

The following example provides an infinite series of non-splitting Rota---Baxter operators on $\A_n$.

\begin{example}\label{exm:An}
Let $m=2$, $q=p$ be a prime, and $p\equiv3\!\pmod4$. There is an infinite number of such primes $p$ and the conditions (i)--(iii) of Theorem~\ref{thm:alternating} hold. 
For every such $p$ we get a~Rota---Baxter operator on $\A_{p^2}$ as in Example~\ref{exm:A9}.

We start with an exact factorization $\A_n = \ker(B)\cdot\Imm(B)=\A_{n-2}\cdot L$, where
$L$ is a sharply 2-transitive group.
Then $H=\Imm(\widetilde{B})=(\Ss_{n-2}\times\Ss_2)\cap\A_n=\langle \A_{n-2}, r \rangle$, where
$r\in R=H\cap L$.
The group~$L$ has two non-isomorphic subgroups of index~2 and we choose one of them 
as $\ker(\widetilde{B})$. Denote $S = \ker(\widetilde{B})$.

Each $x\in\A_n$ has the form $x = kl = kt^\delta y$, where $k \in \ker(B)$, $y\in \ker(\widetilde{B})$ and $\delta\in\{0,1\}$. 
The operator $B$~again acts by the formula~\eqref{eq:ActionBWhen|R|=2}.
Let us clarify the structure of the descendent group~$\A_n^{(\circ)}$~\eqref{R-product}.
We have $\A_n^{(\circ)} = \A_{n-2}\rtimes L = (\A_{n-2}\times S^{\mathrm{op}}).2$,
where for all $h\in \A_{n-2}$, $l\in L$ one computes
$$
l\circ h 
= lB(l)hB(l)^{-1} 
= r^\delta h r^\delta l
= h^{r^\delta}l
= h^{r^\delta}\circ l;
$$
the product $\circ$ on $\A_{n-2}$ coincides with $\cdot$;
the product $\circ$ on the set $S$ is antiisomorphic to $\cdot$, i.\,e. 
$s\circ s' = s's$.
\end{example}
\vspace{0.5em}

The authors can make the GAP code used in this study available upon request.

\section*{Acknowledgements}

The authors are grateful to Andrey Mamontov for the helpful discussions.

The study was supported by a grant from the Russian Science Foundation \textnumero~23-71-10005, https://rscf.ru/project/23-71-10005/

\noindent Alexey Galt \\
Vsevolod Gubarev \\
Novosibirsk State University \\
Pirogova str. 1, 630090 Novosibirsk, Russia \\
Sobolev Institute of Mathematics \\
Acad. Koptyug ave. 4, 630090 Novosibirsk, Russia \\
e-mail: galt84@gmail.com, wsewolod89@gmail.com

\end{document}